\documentclass[10pt, a4paper]{amsart}
\usepackage{amsmath}
\usepackage{amssymb}
\usepackage{color}

\newtheorem{prop}{Proposition}[section]
\newtheorem{teo}{Theorem}[section]
\newtheorem{lema}{Lemma}[section]
\newtheorem{coro}{Corollary}[section]

\theoremstyle{definition}

\def\ep{\varepsilon}

\def\R{\mathbb R}
\def\N{\mathbb N}
\def\K{{\mathcal K}}
\def\D{{\mathcal D}}

\def\H{{\mathcal H}}
\def\U{{\mathcal U}}

\begin{document}
\title[Near field asymptotics for the PME in exterior domains. Critical case]{Near field asymptotics for the porous medium equation in exterior domains. The critical two-dimensional case}

\author[Cort\'{a}zar,  Quir\'{o}s \and Wolanski]{C. Cort\'{a}zar,  F. Quir\'{o}s \and N. Wolanski}

\address{Carmen Cort\'{a}zar\hfill\break\indent
Departamento  de Matem\'{a}tica, Pontificia Universidad Cat\'{o}lica
de Chile \hfill\break\indent Santiago, Chile.} \email{{\tt
ccortaza@mat.puc.cl} }

\address{Fernando Quir\'{o}s\hfill\break\indent
Departamento  de Matem\'{a}ticas, Universidad Aut\'{o}noma de Madrid
\hfill\break\indent 28049-Madrid, Spain.} \email{{\tt
fernando.quiros@uam.es} }

\address{Noem\'{\i} Wolanski \hfill\break\indent
Departamento  de Matem\'{a}tica, FCEyN,  UBA,
\hfill\break \indent and
IMAS, CONICET, \hfill\break\indent Ciudad Universitaria, Pab. I,\hfill\break\indent
(1428) Buenos Aires, Argentina.} \email{{\tt wolanski@dm.uba.ar} }

\thanks{C.\,Cort\'azar supported by  FONDECYT grant 1150028 (Chile). F.\,Quir\'os supported by
project MTM2014-53037-P (Spain). N.\,Wolanski supported by
CONICET PIP625, Res. 960/12, ANPCyT PICT-2012-0153, UBACYT X117 and MathAmSud 13MATH03 (Argentina).}



\date{}

\begin{abstract}
We consider the porous medium equation in an exterior  two-dimensional domain which excludes a hole,  with zero Dirichlet data on  its boundary. Gilding and Goncerzewicz proved in~\cite{Gilding-Goncerzewicz-2007} that in the far field scale, $x=\xi t^{\frac1{2m}}/(\log t)^{\frac{m-1}{2m}}$, $\xi\ne 0$, solutions to this problem with an integrable and compactly supported initial data behave as an instantaneous point-source solution for the equation with a variable mass that decays to 0 in a precise way, determined by the initial data and the hole. However, their result does not say much about the behavior when $|x|=o\big(t^{\frac1{2m}}/(\log t)^{\frac{m-1}{2m}}\big)$, in the so called near field scale, except that the solution is $o\big((t\log t)^{-\frac1m}\big)$ there. In particular, it does not give a sharp decay rate, neither a nontrivial asymptotic profile, on compact sets.
In this paper we characterize the large time behavior in such scale, thus completing the results of~\cite{Gilding-Goncerzewicz-2007}.

\end{abstract}

\maketitle

\section{Introduction}
\label{sect-Introduction} \setcounter{equation}{0}

Let  $\mathcal{H}\subset \mathbb{R}^N$ be a non-empty bounded open set.
We do not assume $\H$ to be connected, so it may represent one or several holes in an otherwise homogeneous medium.  We assume, without loss of generality, that $0\in\mathcal{H}$. Our aim is to complete the description, started by Gilding and Goncerzewicz in~\cite{Gilding-Goncerzewicz-2007}, of the large time behavior of solutions to the Cauchy-Dirichlet problem
\begin{equation}\label{problem}
\tag{P}
\displaystyle\partial_t u=\Delta u^m \quad\mbox{in } (\R^N\setminus\overline\H)\times\R_+,\qquad
u=0\quad\mbox{in }\partial\mathcal{H}\times\R_+,\qquad
u(\cdot,0)=u_0\quad\mbox{in }\mathbb{R}^N\setminus\overline\H,
\end{equation}
with $m>1$, and nonnegative and compactly supported integrable initial data $u_0$, in the critical case $N=2$. A full study of the asymptotic behavior for spatial dimensions $N\ne2$ is already available \cite{Brandle-Quiros-Vazquez-2007,Cortazar-Quiros-Wolanski-preprint,Gilding-Goncerzewicz-2007,Kamin-Vazquez-1991}. This problem, which models the flow of a fluid in a porous medium, has a unique weak solution; see, for instance, the monograph~\cite{Vazquez-book}.

\medskip

\noindent\emph{Notation. } For the sake of brevity we denote $\Omega=\mathbb{R}^N\setminus\overline\H$.

\medskip

\noindent \textsc{The Cauchy problem. }
In the absence of holes, $\H=\emptyset$, the  mass $M(t)=\int_{\mathbb{R}^N}u(\cdot,t)$  of a solution to~\eqref{problem} is conserved, no matter the spatial dimension, $M(t)=\int_{\mathbb{R}^N}u_0$ for all $t\ge0$. Moreover, as proved in~\cite{Friedman-Kamin-1980}, the solution behaves for large times as the instantaneous point-source solution  of the equation with mass $M=\int_{\mathbb{R}^N}u_0$, that we denote by $\U(x,t;M)$, in the following precise sense,
\begin{equation}
\label{eq:convergence.Cauchy.problem}
\lim_{t\to\infty}t^{\frac{N}{N(m-1)+2}}\|u(\cdot,t)-\U\big(\cdot,t;M\big)\|_{L^\infty(\mathbb{R}^N)}=0,\qquad M=\int_{\mathbb{R}^N}u_0;
\end{equation}
see also~\cite{Vazquez-2003} and the references therein. This result does not require the initial data to be compactly supported.
The special solution $\U(x,t;M)$, which was discovered by Zel'dovi$\check{\rm c}$ and Kompaneets~\cite{Zeldovic-Kompaneec-1950} in dimensions one and three, and by Barenblatt~\cite{Barenblatt-1952} and Pattle~\cite{Pattle-1959} for arbitrary dimensions, has $M\delta$  (where $\delta$ is the Dirac distribution) as initial data.  It has a selfsimilar form,
\begin{equation}\label{eq-U}
\U(x,t;M)=t^{-\alpha}F_M(\xi),\quad \xi=\frac{x}{t^\beta},\qquad\beta=\frac{1}{N(m-1)+2},\quad \alpha=N\beta,
\end{equation}
with a profile
\begin{equation}
\label{eq:profile}
\left\{
\begin{array}{l}
F_M(\xi)=\Big(\frac{(m-1)\beta}{2m}\Big)^{\frac1{m-1}}
(\xi_M^2-|\xi|^2)_+^{\frac1{m-1}},\\[8pt] \xi_M=\left(
\frac{\Gamma\left(\frac{1}{2(m-1)\beta}\right)}{4m\pi^{\frac N2}\Gamma\left(\frac{m}{m-1}\right)}\right)^{(m-1)\beta}\left(\frac{2m}{(m-1)\beta}\right)^{m\beta}M^{(m-1)\beta}.
\end{array}
\right.
\end{equation}
Notice that, for all times, $\U(\cdot,t;M)$ has a compact support, namely $\{x\in\mathbb{R}^N: |x|\le \xi_M t^\beta\}$. This property, known as finite speed of propagation, is shared by all solutions having an integrable initial data with compact support.

Let us remark that the scaled variable
$$
w(\xi,\tau)=t^{\alpha}u(\xi t^\beta,t), \qquad t=\textrm{e}^{\tau},
$$
satisfies the non-linear Fokker-Planck type equation
\begin{equation}
\label{eq:Fokker-Planck}
\partial_\tau w(\xi,\tau)=\Delta w^m(\xi,\tau)+\beta \nabla\cdot\big(\xi w(\xi,\tau)\big),
\end{equation}
whose stationary (integrable) states are precisely the profiles $F_M$.
When written in terms of this scaled variable, the convergence result~\eqref{eq:convergence.Cauchy.problem} just says that $w$ converges uniformly as $t\to\infty$ ($\tau\to\infty$) towards a stationary state of~\eqref{eq:Fokker-Planck}, which one being dictated by the conservation of mass.

\medskip

\noindent\textsc{A conservation law.  }
In the presence of holes, solutions of~\eqref{problem} do not conserve mass. However, we still have an invariant of the evolution, namely
\begin{equation}
\label{eq:conservation.law}
M_\phi(t):=\int_{\Omega}u(\cdot,t)\phi=\underbrace{\int_{\Omega}u_0\phi}_{M_\phi^*}\quad\text{for all }t>0,
\end{equation}
where $\phi\in C(\overline\Omega)\cap C^2(\Omega)$ satisfies
\begin{equation}
\label{eq:stationary.general}
\Delta\phi=0\quad\text{in }\Omega,\qquad\phi=0\quad\text{on }\partial\Omega,
\end{equation}
plus some prescribed behavior at infinity;
see~\cite{King-1991} for a first proof of this fact in the radial case, and also~\cite{Brandle-Quiros-Vazquez-2007} and~\cite{Gilding-Goncerzewicz-2007}.
In order to have a unique nontrivial solution to problem~\eqref{eq:stationary.general}, not all behaviors at infinity are allowed. Which ones are possible depend on $N$. This will lead to different asymptotic behaviors for solutions to problem~\eqref{problem} depending precisely on the dimension.

\medskip

\noindent\textsc{The problem in high dimensions. } When $N\ge3$, in order to have a nontrivial solution to~\eqref{eq:stationary.general} we have to ask $\phi(x)$ to approach a constant as $|x|\to\infty$. This constant is chosen to be equal to one for simplicity.   With this choice for $\phi$, the conservation law~\eqref{eq:conservation.law} yields that the  mass $M(t):=\int_\Omega u(\cdot,t)$ converges to the nontrivial value $M_\phi^*$. Moreover, as proved in~\cite{Brandle-Quiros-Vazquez-2007}, the asymptotic behavior is given by
$$
\lim_{t\to\infty}t^{\frac{N}{N(m-1)+2}}\|u(\cdot,t)-\phi^{\frac1m}\U\big(\cdot,t;M_\phi^*\big)\|_{L^\infty(\mathbb{R}^N)}=0;
$$
see also~\cite{Herraiz-1999} for the linear case, $m=1$. Thus, in the  \emph{far field scale},  $x=\xi t^\beta$, $\xi\ne0$,  we have convergence towards the instantaneous point-source solution with mass $M_\phi^*$. More precisely, the \emph{outer behavior} is given by
$$
t^{\alpha}u(\xi t^\beta,t)\to F_{M_\phi^*}(\xi)\quad\text{uniformly for } 0<\xi_1\le |\xi|\le\xi_2<\infty,
$$
with $\alpha$ and $\beta$ as in~\eqref{eq-U};
see also~\cite{Gilding-Goncerzewicz-2007}.  On the other hand, the behavior in inner regions, $|x|=o(t^\beta)$, is given in terms of the stationary solution,
$$
t^{\alpha}u(x,t)\to F_{M_\phi^*}(0)\phi^{\frac1m}(x)\quad\text{uniformly on compact sets}.
$$
Note that the two developments overlap in a wide region,
$$
x=\xi g(t), \quad\xi\ne0, \quad g(t)\to\infty, \quad g(t)=o(t^\beta),
$$
where $u(x,t)\to F_{M_\phi^*}(0)$. In fact, the coefficient multiplying $\phi^{\frac1m}$ in the inner region is obtained from the outer limit by \emph{matching} both expansions in the overlapping region.

Let us remark that the rate of decay coincides, both in the inner and in the outer region, with that of the Cauchy problem. This is also true for the expansion rate; see the outer behavior. This is more clearly seen when studying the long time behavior of the support. Indeed, if
\[\begin{aligned}
&\zeta_+(t)=\sup\{\rho>0: u(x,t)>0\ \mbox{for some }x\in\overline\Omega, \ |x|=\rho\},\\
&\zeta_-(t)=\inf\{\rho>0:u(x,t)=0\ \mbox{for some }x\in\Omega,\ |x|=\rho\},
\end{aligned}
\]
then, as proved in~\cite{Brandle-Quiros-Vazquez-2007},
$
\zeta_\pm(t)/t^{\beta}\to \xi_{M_\phi^*}$,
where $\xi_M$ has the same meaning as in~\eqref{eq:profile}. An analogous result holds for the Cauchy problem; see for instance~\cite{Vazquez-2003}. This means in particular that $u(x,t)$ is identically zero  for  $|x|>\xi_M t^\beta$ and large times. This includes the \emph{very far field} scale, $|x|=t^\beta g(t)$, $g(t)\to\infty$.



\medskip

\noindent\textsc{The problem in the half-line. } In dimension one, if there is a hole the mass goes to zero, $M(t)=O(t^{-1/(2m)})$. Hence, we expect decay and expansion rates different from that of the Cauchy problem.

In this one-dimensional case, a hole disconnects the domain in several  components, two of them unbounded. After a translation and maybe a reflection, the problem in these unbounded components can be transformed into the problem in the half-line $\mathbb{R}_+$.  Nontrivial  harmonic functions on the half-line with zero boundary value are  multiples of $x$.  Therefore,  the conservation law means in this case that  solutions have a constant first moment, which indicates what is the right scaling to study the large time behavior: one that preserves that quantity. Thus,  solutions approach a so-called \emph{dipole} solution $\mathcal{D}$ to the equation with the same first moment as $u_0$. More precisely, as proved in~\cite{Kamin-Vazquez-1991},
\begin{equation}
\label{eq:conv.KV}
\lim_{t\to\infty}t^{\frac1m}\|u(\cdot,t)-\mathcal{D}(\cdot,t;M_\phi^*)\|_{L^\infty(\mathbb{R}_+)}=0,\qquad M_\phi^*=\int_0^\infty xu_0(x)\,dx.
\end{equation}
The special solution $\mathcal{D}$, discovered by Barenblatt and Zel'dovi$\check{\rm c}$~\cite{Barenblatt-Zeldovich-1957}, has a self-similar structure,
\begin{equation*}
\label{eq:selfsimilar.form.dipole}
\mathcal{D}(x,t;M)=t^{-\alpha_{\rm d}}D_{M}(\xi), \qquad\xi=x/t^{\beta_{\rm d}},\quad \alpha_{\rm d}=\frac1m,\ \beta_{\rm d}=\frac1{2m},
\end{equation*}
with a profile
$$
\left\{
\begin{array}{l}
D_M(\xi)=\left(\frac{m-1}{2m(m+1)}\right)^{\frac1{m-1}}\xi^{\frac1m}\Big(\xi_M^{\frac{m+1}m}
-\xi^{\frac{m+1}m}\Big)_+^{\frac1{m-1}},
\\[10pt]
\xi_M=\left(\frac{2m(m+1)}{(m-1)\big(\int_0^1 s^{\frac{m+1}{m}}(1-s^{\frac{m+1}{m}})^{\frac1{m-1}}\,ds\big)^{m-1}}\right)^{\frac1{2m}} M^{\frac{m-1}{2m}},
\end{array}
\right.
$$
and, due to the choice of the similarity exponents $\alpha_{\rm d}$ and $\beta_{\rm d}$, its first moment is constant in time. Note that $\mathcal{D}$ has compact support in space for all times.

Though the convergence result~\eqref{eq:conv.KV} is valid in the whole half-line, it only gives the exact decay rate and a nontrivial asymptotic profile in the far field scale, $x=\xi t^{\beta_{\rm d}}$, $0<\xi<\xi_{M_\phi^*}$, since $t^{\alpha_{\rm d}}\mathcal{D}(g(t),t;M_\phi^*)=0$ if $g(t)=o(t^{\beta_{\rm d}})$. The behavior in the very far field was also determined in~\cite{Kamin-Vazquez-1991}, where the authors prove that
$s(t)=\sup\{x:u(x,t)>0\}$ satisfies $s(t)/t^{\beta_{\rm d}}\to\xi_{M_\phi^*}$.

As for the behavior in the near field, the dipole solution gives a hint of the right scaling. Indeed, the decay rate of $\mathcal{D}$ for $x=g(t)$ with $g(t)=o(t^{\beta_{\rm d}})$ is $O(t^{-\alpha_{\rm d}-\frac{\beta_{\rm d}}{m}}(g(t))^{1/m})$. Having this in mind, we proved in~\cite{Cortazar-Quiros-Wolanski-preprint}, by a careful matching with the outer behavior, that
$$
\lim_{t\to\infty} t^{\alpha_{\rm d}+\frac{\beta_{\rm d}} m}\sup_{x\in\mathbb{R}_+}\frac{\big|u(x,t)-\mathcal{D}(x,t;M_\phi^*)\big|}{(1+x)^{\frac1 m}}=0,
$$
which settles down the long time asymptotics in the near field scale; see also~\cite{Cortazar-Elgueta-Quiros-Wolanski-2015} for the linear case. It may seem at first glance that stationary solutions do not play a role here. But they are still there, hidden in the dipole solution, since
$$
D_M(\xi)\approx\left(\frac{m-1}{2m(m+1)}\right)^{\frac1{m-1}}\xi_M^{\frac{m+1}{m(m-1)}}
\phi(\xi)^{\frac1m}\quad\text{for }\xi\approx0.
$$
Thus, in particular,
$$
t^{\alpha_{\rm d}+\frac{\beta_{\rm d}}{m}}u(x,t)\to \left(\frac{m-1}{2m(m+1)}\right)^{\frac1{m-1}}\xi_{M_\phi^*}^{\frac{m+1}{m(m-1)}}x^{1/m}\quad\text{uniformly on compact sets}.
$$
Note that the decay rate, which differs from that of the Cauchy problem, depends on the scale, and is given by the ratio $x^{1/m}/t^{\alpha_{\rm d}+\frac{\beta_{\rm d}}{m}}$.  This makes the matching quite involved, since the overlapping region between the inner and the outer behavior is very narrow.

\medskip

\noindent\textsc{The critical case: Outer behavior. } In the critical two-dimensional case, the behavior at infinity leading to nontrivial stationary solutions is logarithmic. Thus, we have to look for $\phi\in C(\overline\Omega)\cap C^2(\Omega)$ satisfying
\begin{equation}
\label{eq:stationary.problem}
\tag{S}
\Delta\phi=0\quad\text{in }\Omega\subset \mathbb{R}^2,\qquad\phi=0\quad\text{on }\partial\Omega,\qquad |\phi(x)-\log|x||\le C \quad\text{for all } x\in\overline\Omega.
\end{equation}
There is a unique such function if, for instance, $\mathcal{H}$ has the interior tangent ball property; see Section~\ref{sect:stationary}.
As shown in~\cite{Gilding-Goncerzewicz-2007}, the conservation law~\eqref{eq:conservation.law} implies then  on the one hand the global decay rate
\begin{equation}
\label{eq:global.rate}
\|u(\cdot,t)\|_{L^\infty(\mathbb{R}^2)}=O\big((t\log t)^{-\frac1m}\big)\quad\text{as } t\to\infty,
\end{equation}
and on the other hand that the mass  satisfies
\begin{equation}
\label{eq:mass.behavior}
\lim_{t\to\infty}\log t\,M(t)=2m M_\phi^*.
\end{equation}
Using these two facts, Gilding and Goncerzewicz proved in~\cite{Gilding-Goncerzewicz-2007}, in the case in which $\H$ is $C^{2,\alpha}$ and simply connected,   that
\begin{equation}\label{eq:far.field.limit}
\lim_{t\to\infty}(t\log t)^{ 1/m}\sup_{x\in\mathcal{O}_\delta(t)}|u(x,t)-\U\big(x,t;2m M_\phi^*/\log t\big)\big|=0 \quad\text{for any }\delta>0,
\end{equation}
in outer sets
\begin{equation*}
\label{eq:def.outer.sets}
\mathcal{O}_\delta(t)=\big\{x\in\Omega: |x|\ge \delta t^{\frac1{2m}}(\log t)^{-\frac{m-1}{2m}}\big\},\quad t\ge1.
\end{equation*}
In addition, they also proved that
\begin{equation}
\label{eq:limit.support}
t^{-\frac1{2m}}(\log t)^{\frac{m-1}{2m}}\zeta_\pm(t)\to \xi_{2mM_\phi^*} \quad\text{as }t\to\infty,
\end{equation}
where $\xi_M$ has the same meaning as in~\eqref{eq:profile}.
We will extend these results to the case in which $\H$ is not necessarily simply connected in Section~\ref{sect:outer.general.holes}.

It is worth noticing that the appropriate scaled function in this setting,
$$
w(\xi,\tau)=(t\log t)^{\frac1m}u\Big(\xi t^{\frac1{2m}}(\log t)^{-\frac{m-1}{2m}},t\Big), \qquad t=\textrm{e}^{\tau},
$$
does not satisfy the non-linear Fokker-Planck type equation~\eqref{eq:Fokker-Planck}, but a perturbation,
$$
\partial_\tau w(\xi,t) =\Delta w^m(\xi,t)+\frac{\nabla \cdot\big(\xi w(\xi,t)\big)}{2m}\Big(1+\frac1\tau\Big)-\frac{\xi\cdot\nabla w(\xi,\tau)}{2\tau}.
$$
However, in the limit $t\to\infty$ ($\tau\to\infty$) the extra terms become negligible, and we have, also in this case, convergence towards a stationary solution to~\eqref{eq:Fokker-Planck}, which one being given by the conservation law~\eqref{eq:conservation.law}.

Note that, though  the function giving the asymptotic behavior in outer sets  is not a solution of the equation, it still has a selfsimilar structure,
\begin{equation}
\label{eq:limit.is.selfsimilar}
\U(x,t;2mM_\phi^*/\log t)=(t\log t)^{-\frac1m} F_{2mM_\phi^*}(\tilde\xi), \qquad \tilde\xi=\frac{x(\log t)^{\frac{m-1}{2m}}}{t^{\frac1{2m}}}.
\end{equation}
In comparison with the case with no holes, we find logarithmic corrections both in the decay and in the expansion rate.

\medskip

\noindent\textsc{The critical case: Inner behavior. }
The main aim of the present paper is to fill the gap in~\cite{Gilding-Goncerzewicz-2007} by describing the behavior in inner sets of the form
\begin{equation}
\label{eq:def.inner.sets}
\mathcal{I}_\delta(t)=\{x\in\Omega: |x|\le\delta t^{\frac1{2m}}(\log t)^{-\frac{m-1}{2m}}\},\quad t\ge1.
\end{equation}
Due to some details of our technique, we will need to have $\nabla \phi\ne0$. As we will see later, this condition is fulfilled for instance if $\H$ is simply connected and $C^{1,\alpha}$ smooth.
\begin{teo}
\label{thm:main} Let $\H\ni0$ be a bounded open subset of $\mathbb{R}^2$ with $C^{1,\alpha}$ boundary such that the unique solution $\phi$  to~\eqref{eq:stationary.problem} satisfies $\nabla\phi\neq0$ in $\overline\Omega$. If $u$ is a solution to \eqref{problem}, then,  for all $\delta\in(0,\delta_*)$, $\delta_*=\xi_{2mM_\phi^*}\sqrt{(m-1)/(2m)}$,
\begin{equation}\label{eq-asym-con-phi}
\lim_{t\to\infty}(t\log^2 t)^{\frac1m}\sup_{x\in\mathcal{I}_\delta(t)}{\frac{\Big|u(x,t)-\Big(\frac{2m\phi(x)}
{\log t}\Big)^{\frac1m}\U\big(x,t;2m M_\phi^*/\log t\big)\Big|}{\big(\log(|x|+\mbox{\rm e})\big)^{\frac1m}}}=0.
\end{equation}
\end{teo}
Note that the rate of decay  is given by the ratio  $\left(\frac{\log |x|}{t\log^2 t}\right)^{\frac1m}$.  Thus, there is a continuum of possible decay rates, starting with
the decay rate $O((t\log t)^{-\frac1m})$, holding in the far field, all the way up to $O((t\log^2 t)^{-\frac1m})$, that takes
place on compact sets.
Moreover, the scaled function $(t\log^2 t/\log(|x|+\text{\rm e}))^{\frac1m}u(x,t)$ converges in the near field scale, $|x|\le t^{1/2m}g(t)$, $\lim_{t\to\infty}g(t)=0$, to a multiple of $(\phi(x)/\log(|x|+\text{\rm e}))^{\frac1m}$, where $\phi$ is the unique solution to~\eqref{eq:stationary.problem}. In particular,
$$
(t\log^2 t)^{\frac1m}u(x,t)\to\Big(\frac{m M_\phi^*}{\pi}\Big)^{\frac1m} \, (\phi(x))^{\frac1m}\quad  \text{uniformly on compact (in $x$) sets}.
$$

Since $2m\phi(x)/\log t\to 1$ as $t\to\infty$ in outer regions $\mathcal{O}_\delta(t)$, and taking also into account that the support grows as $O(t^{\frac1{2m}}(\log t)^{-\frac{m-1}{2m}})$, we can express the asymptotic behavior in all scales in a unified way.
\begin{coro}
Under the hypotheses of Theorem~\ref{thm:main}
\begin{equation*}\label{eq:global.asymptotics}
\lim_{t\to\infty}(t\log^2 t)^{\frac1m}\sup_{x\in\mathbb{R}^2}{\frac{\Big|u(x,t)-\Big(\frac{2m\phi(x)}
{\log t}\Big)^{\frac1m}\U\big(x,t;2m M_\phi^*/\log t\big)\Big|}{\big(\log(|x|+\mbox{\rm e})\big)^{\frac1m}}}=0.
\end{equation*}
\end{coro}
This result is still valid in the linear case $m=1$, taking $\U(x,t;M)=\frac{M\textrm{e}^{-\frac{|x|^2}{4t}}}{4\pi t}$. This was proved in~\cite{Herraiz-1999} by means of a representation formula for the solution of the problem in terms of the instantaneous point-source solution (the case of high dimensions, $N\ge3$, was also treated there). In our non-linear setting, such a formula is not available, and we have to use an alternative approach, based in comparison with carefully chosen sub and supersolutions combined with a matching with the outer behavior.  The matching is quite involved, since the rates of decay in inner and outer regions are not the same. We already had this difficulty in~\cite{Cortazar-Elgueta-Quiros-Wolanski-2016}, where we dealt with a two-dimensional nonlocal linear heat equation.

\medskip

\noindent\textsc{The problem in bounded domains. } In any dimension, since we are not asking the hole to be connected, $\Omega$ may have bounded components, where $\phi$ is identically zero, which means that
the above mentioned convergence results do not give the right scalings and profiles there. Indeed, as proved in~\cite{Aronson-Peletier-1981}, no matter the dimension,  nonnegative and nontrivial solutions to the porous medium equation in a bounded set $\Omega_{\rm b}$ with zero Dirichlet data satisfy
$$
\lim_{t\to\infty}t^{\frac1{m-1}}\|u(\cdot,t)-U_{\rm b}(\cdot,t)\|_{L^\infty(\Omega_{\rm b})}=0,
$$
where $U_{\rm b}(x,t)=t^{-\frac1{m-1}}f(x)$ is a solution in separated variables with a profile  $f$ which is the unique nonnegative solution to the non-linear elliptic problem
$$
\Delta f^m+f=0\quad\text{in }\Omega_{\rm b},\qquad f=0\quad\text{on }\partial\Omega_{\rm b};
$$
see also~\cite{Vazquez-2004} and the references therein. Notice that the solution decays faster in this bounded components, and that the asymptotic shape and size are universal, not depending on the initial data.

\medskip

\noindent\textsc{Non-trivial boundary data. }
When the boundary data are stationary and nontrivial, $u(x,t)=g(x)$, $x\in\partial\Omega$, $g\ge0$, $g\not\equiv0$, the asymptotic behavior in exterior domains is quite different; see~\cite{Peletier-1971} for the case of the half-line and~\cite{Quiros-Vazquez-1999} for higher dimensions. Solutions with integrable initial data asymptotically gain mass.

On the half-line the behavior (in all scales) is given by the unique (self-similar) solution to
the problem with trivial initial data, which has the form $U_g(x,t)=f_g(x/t^{1/2})$. Observe that $g$ is a constant in this case.

As for higher dimensions, on compact sets solutions converge uniformly to $\phi^{1/m}$, where now $\phi$ is the solution to
$$
\Delta \phi=0\quad\text{in }\Omega, \qquad \phi=g^m\quad\text{in }\partial\Omega,\qquad \lim_{|x|\to\infty}\phi(x)=0\text{ if } N\ge3,\quad\phi\text{ bounded if }N=2,
$$
and  the outer behavior is given in terms of the solution $\mathcal{U}_{\rm c}$  to
$$
\partial_t \mathcal{U}_{\rm c}=\Delta \mathcal{U}_c^m+c\delta \quad\text{in }\mathcal{D}'(\mathbb{R}^N\times\mathbb{R}_+),\qquad \mathcal{U}_{\rm c}(\cdot,0)=0 \quad \text{a.e.\,in }\mathbb{R}^N,
$$
with logarithmic corrections in the critical case $N=2$. This special solution has a self-similar form, $\mathcal{U}_{\rm c}(x,t)=t^{-\alpha}F_{\rm c}(x/t^\beta)$, with similarity exponents $\alpha=(N-2)/(N(m-1)+2)$, $\beta=m/(N(m-1)+2)$. The right constant $c$ is inherited from the inner behavior through a matching procedure.

\medskip

\noindent\textsc{Organization of the paper. } We devote Section~\ref{sect:stationary} to the analysis of the stationary problem~\eqref{eq:stationary.problem}. Next, in Section~\ref{sect:outer.general.holes}, we extend the results on the outer behavior from~\cite{Gilding-Goncerzewicz-2007} to more general holes than those considered there. Convenient super and subsolutions with the adequate large time behavior are constructed respectively in Sections~\ref{sect:control.above} and~\ref{sect:control.below}.

\medskip

\noindent\emph{Notations. }  In what follows, $B_r=\{x\in\mathbb{R}^2: |x|<r\}$,  $F_*=F_{2mM_\phi^*}$, $\xi_*=\xi_{2mM_\phi^*}$, and $G(x,t)=\U(x,t;2mM_\phi^*/\log t)$.

\section{The stationary problem}
\label{sect:stationary}
\setcounter{equation}{0}

This section is devoted to  studying the stationary problem~\eqref{eq:stationary.problem}. Existence and uniqueness were already proved in~\cite{Gilding-Goncerzewicz-2007} when the boundary of the hole is $C^{2,\alpha}$ smooth, using some classical results that can be found for instance in~\cite{Benilan-1990}. This smoothness assumption can be relaxed. As we prove next, it is enough to have the \emph{interior tangent ball property}:
\begin{equation}
\label{eq:interior.tangent.ball.property}
\forall \bar x\in\partial\H\text{ there are } x_0\in \H \text{ and } r_0>0 \text{ such that } \bar x\in\partial B_{r_0}(x_0)\text{ and }B_{r_0}(x_0)\subset \H.
\end{equation}
\begin{prop}
\label{prop:stationary}
Let $\H\ni0$ be a bounded, open subset of $\mathbb{R}^2$ with the interior tangent ball property~\eqref{eq:interior.tangent.ball.property}. There exists a unique solution to~\eqref{eq:stationary.problem}.
\end{prop}
\begin{proof} Take $r,R>0$ such that
$B_r\subset \H\subset B_{R}$. For each $n\in\N$, $n>R$, let $\phi_n$ be the harmonic function in $B_n\setminus\overline\H$ such that $\phi_n\equiv 0$ on $\partial\H$ and $\phi_n(x)=\log\frac{|x|}R$ on $\partial B_n$. A simple comparison argument shows that $\log\frac{|x|}R\le\phi_n(x)\le \phi_{n+1}(x)\le\log\frac{|x|}r$ in $B_n\setminus\H$. Therefore, there exists a function  $\phi$ such that $\phi_n(x)\nearrow\phi(x)$ as $n\to\infty$ for every $x\in \overline\Omega$. Moreover, $\log\frac{|x|}R\le \phi(x)\le \log \frac{|x|}r$, which implies that $\phi(x)-\log |x|$ is bounded. On the other hand, by Harnack's Convergence Theorem, $\phi$ is harmonic in $\Omega$.

Since $\phi_n=0$ on $\partial\H$ for every $n$, the same property holds for $\phi$. Let us now prove that $\phi$ is continuous on $\partial\H$. Given $\bar x\in\partial\H$, let $x_0\in \H$ and  $r_0>0$ be as in~\eqref{eq:interior.tangent.ball.property}. We can assume without loss of generality that $r_0<r$.  Then, there exists a value $\kappa>0$ such that for $n\ge \kappa$ there holds that $\phi_n(x)=\log \frac{|x|}R\le 2\log\frac{|x-x_0|}{r_0}$  if $x\in\partial B_n$. Hence,  the comparison principle yields that $\phi_n(x)\le 2\log\frac{|x-x_0|}{r_0}$ for all $x\in B_n\setminus\H$. We conclude that $\phi(x)\le \log\frac{|x-x_0|}{r_0}$ in $\overline\Omega$, so that $\phi(x)\to0$ as $x\to \bar x$.

In order to prove uniqueness we let $\phi_1$ and $\phi_2$ be two solutions and $\psi=\phi_1-\phi_2$. Let $\K\subset\subset \overline\Omega$ and $R>0$ large so that $\K\subset B_R$. Let $\ep>0$. If $R$ is large enough, $|\phi_i(x)-\log |x||<\ep\log \frac{|x|}r$ on $\partial B_R$, where $r>0$ is as before. Therefore, by the comparison principle, $\psi(x)\le 2\ep\log\frac{|x|}r$ in $B_R\setminus\H\supset \K$. Letting $\ep\to0$ we get that $\psi\le 0$ in $\K$.

By reversing the roles of $\phi_1$ and $\phi_2$ we obtain that $\phi_1=\phi_2$.
\end{proof}

%

In order to prove our main result, Theorem~\ref{thm:main}, we need $|\nabla \phi|$ to be comparable to $1/|x|$. We start by checking that this is true \lq\lq at infinity''.
\begin{prop} Assume the hypotheses of Proposition~\ref{prop:stationary}. There exists  $R>0$ such that the unique solution $\phi$ to~\eqref{eq:stationary.problem} satisfies
\begin{equation}\label{eq-bound-nabla-phi}
\frac12 \le x\cdot\nabla\phi(x)\le |x||\nabla\phi(x)|\le 2\quad\mbox{for }|x|\ge R.
\end{equation}
\end{prop}
\begin{proof}Let $r>0$ such that $B_{2r}\subset\H$.
For $k\in\N$, let $\psi_k(x)=\phi(kx)-\log\big(\frac{k|x|}r\big)$. Then, $\psi_k$ is harmonic in $\R^2\setminus\overline\H_k$ where $\H_k=\{x\in\R^2:kx\in\H\}$. Moreover, there exists $C>0$ such that $|\psi_k(x)|\le C$ in $\R^2\setminus\overline\H_k$. Hence, for every sequence $\{k_n\}$ with $k_n\to\infty$ there is a subsequence $\{k_{n_j}\}$ and a harmonic and bounded function $\psi$ such that $\psi_{k_{n_j}}\to\psi$ as $j\to\infty$ uniformly on compact subsets of $\R^2\setminus\{0\}$. Since $\psi$ is harmonic and bounded in $\R^2\setminus\{0\}$, it can be extended to a harmonic and bounded function in $\R^2$, which, by Liouville's Theorem, is constant, $\psi\equiv C_0$.

Let $\zeta_j=\psi_{k_{n_j}}-C_0$. Since $\zeta_j\to0$ uniformly on compact subsets of $\R^2\setminus\{0\}$, then  $|\nabla\zeta_j|\to0$ uniformly in $B_2\setminus B_{\frac12}$. Since the limit is independent of the subsequence, convergence is not restricted to subsequences, and hence
\[
k\nabla\phi(kx)\to \frac x{|x|^2}\quad\mbox{uniformly for } \frac12\le |x|\le 2\mbox{ as }k\to\infty.
\]
Therefore, taking $y=kx$ with $|x|=1$, which means $k=|y|$, we conclude that
\begin{equation}\label{eq-bound-ynablaphi}
|y|\nabla\phi(y)\to \frac y{|y|}\quad\mbox{as }|y|\to\infty.
\end{equation}
Estimate~\eqref{eq-bound-nabla-phi} now follows easily by observing that
$$
x\cdot\nabla\phi(x)=1+\frac x{|x|}\cdot\left(|x|\nabla \phi(x) -\frac x{|x|}\right),
$$
since we know from~\eqref{eq-bound-ynablaphi} that $\left||x|\nabla \phi(x) -\frac x{|x|}\right|\le1/2$ if $|x|\ge R$ for some large $R$.
\end{proof}
If $\partial\H\in C^{1,\alpha}$, then $\phi\in C^1(\overline\Omega)$, and  we deduce that $|x||\nabla \phi(x)|\le C$ in $\overline\Omega$. In order to proceed further in this smooth case, we still have to prove that $|x||\nabla \phi(x)|\ge c>0$ in bounded sets. Since $|\nabla\phi|\ge c_0>0$ in a neighborhood of $\partial\H$, we only need to check that $\nabla\phi\neq0$ in $\Omega$. This is the case if, for instance, in addition to being smooth, $\H$ is simply connected, as we prove next.
\begin{prop}
\label{prop:nontrivial.gradient}
Let $\H\ni0$ be a bounded, simply connected, open subset of $\mathbb{R}^2$ with $C^{1,\alpha}$ boundary. Let $\phi$ be the unique solution to~\eqref{eq:stationary.problem}. Then, $\nabla\phi\neq0$ in $\overline\Omega$.
\end{prop}
\begin{proof} We will get a semi-explicit formula for the solution $\phi$, from where the result follows immediately.

Let $r>0$ be such that $B_r\subset\H$. First, we perform an inversion with respect to $B_r$. The \lq inverted' set
\[
\Omega'=\Big\{x=\frac{r^2 y}{|y|^2}:y\in \Omega\Big\}\cup\{0\}
\]
is a simply connected open set contained in $B_r$. Moreover, since $\partial\H$ is a Jordan arc and $0\not\in\partial\H$, then
\[\partial\Omega'=\Big\{x=\frac{r^2 y}{|y|^2}:y\in \partial\H\Big\}
\]
is also a Jordan arc. The Kelvin transform $\psi(x)=\phi\big(\frac{r^2 x}{|x|^2}\big)$  is well defined and harmonic in $\Omega'\setminus\{0\}$,  $\psi=0$ on $\partial\Omega'$, and $\psi(x)-\log\frac1{|x|}$  bounded.

%
%
%
%
%
%

%

Let now $f:B_1\to\Omega'$ be a conformal mapping. This is, $f$ is holomorfic in $B_1$, continuous in $\overline B_1$, $f(\partial B_1)=\partial\Omega'$, $f$ is one to one between $B_1$ and $\Omega'$, $f'(x)\neq0$ for $x\in B_1$, $f(0)=0$ and ${\mathfrak a}:=f'(0)\in\mathbb{R}$, $\mathfrak{a}>0$. The existence (and uniqueness) of such a mapping is guaranteed by Riemann's Mapping Theorem. Let
$$
\varphi(x)=\psi(f(x))=\phi\big({\textstyle\frac{r^2 f(x)}{|f(x)|^2}}\big)\quad\mbox{for }x\in B_1\setminus\{0\}.
$$
Since $f(0)=0$ and ${\mathfrak a}=f'(0)>0$, there exists $\ep>0$ such that $2{\mathfrak a}|x|\ge|f(x)|\ge \frac {\mathfrak a}2|x|$ if $|x|\le\ep$. Thus, since $f$ is one to one, there exists $\delta>0$ such that $\delta<|f(x)|<r$ if $\ep<|x|<1$. Hence,
\[\Big|\log{\textstyle\frac1{|f(x)|}}-\log{\textstyle\frac 1{{\mathfrak a}|x|}}\Big|
=\left|\log {\textstyle\frac{|f(x)|}{{\mathfrak a}|x|}}\right|\le C\quad\mbox{for } 0< |x|<1.
\]
On the other hand,
for $x\in B_1\setminus\{0\}$,
\[
\Big|\varphi(x)-\log{\textstyle\frac1{|f(x)|}}\Big|=\big|\phi({\textstyle
\frac{r^2 f(x)}{|f(x)|^2}}\big)-\log{\textstyle\frac1{|f(x)|}}\Big|=\big|\phi(y)-\log|y|+2\log r\big|\quad
\mbox{with }y={\textstyle\frac{r^2f(x)}{|f(x)|^2}}\in \overline\Omega,
\]
and hence  $\varphi(x)-\log\frac 1{|x|}$ is bounded in $B_1\setminus\{0\}$. Therefore, since $\varphi$ is harmonic in $B_1\setminus\{0\}$ and $\varphi=0$ on $\partial B_1$, we deduce that $\varphi(x)-\log\frac1{|x|}=0$ in $B_1\setminus\{0\}$.
Thus, $\phi(x)=-\log \big|f^{-1}\big(\frac{r^2 x}{|x|^2}\big)\big|$.
In particular, $\nabla\phi\neq 0$ in $\overline\Omega$.
\end{proof}

\begin{coro}
Under the hypotheses of Proposition~\ref{prop:nontrivial.gradient}, there exist constants $C,c>0$ such that
$$
c\le |x| |\nabla \phi(x)|\le C.
$$
\end{coro}

\section{Outer behavior for general holes}
\label{sect:outer.general.holes}
\setcounter{equation}{0}

The aim of this section is to extend the results of~\cite{Gilding-Goncerzewicz-2007} on the exterior behavior and the global decay to the general case in which $\H$ is not necessarily simply connected and smooth.
The proof is based on the result of \cite{Gilding-Goncerzewicz-2007} for simply connected smooth holes.

\begin{teo}\label{teo-asym-ext-phi} Let $\H\ni0$ be a bounded, open subset of $\mathbb{R}^2$ and $u$ a solution to \eqref{problem}. Then, equations~\eqref{eq:global.rate}--\eqref{eq:limit.support} hold.
\end{teo}
\begin{proof}
Let $r>0$ be such that $\overline B_r\subset\H$ and let $u_r$ be the solution to \eqref{problem} with a hole $\H_r=B_r$ and initial condition
\[
u_r(x,0)=\begin{cases} u_0(x),\quad&x\in\Omega,\\
                       0,\quad&x\in\H\setminus B_r.
\end{cases}
\]
Then, $u\le u_r$ in $\overline\Omega\times \mathbb{R}_+$, so that, by the results of \cite{Gilding-Goncerzewicz-2007},  we have~\eqref{eq:global.rate} and moreover
\begin{equation}
\label{eq:decay.mass}
\int_\Omega u(\cdot,t)\le \frac{C}{\log t}.
\end{equation}

Given $r$ as above and $R$ large so that $\overline\H\subset B_R$, let $\Omega^-=\R^2\setminus B_R$ and $\Omega^+=\R^2\setminus B_r$. For any $T>0$ we define $u^-(x,t;T)$ and $u^+(x,t;T)$ respectively as the solutions to \eqref{problem} in $\Omega^\pm\times\mathbb{R}_+$ with initial data
\[
u^-(x,0;T)=u(x,T)\quad\text{in }\Omega^-, \qquad
u^+(x,0;T)=\begin{cases} u(x,T),\quad&x\in\Omega,\\
                       0,\quad&x\in\Omega_+\setminus\Omega.
\end{cases}
\]
Then, by a straightforward comparison argument,
\begin{equation}\label{eq-comparison-asym}
u^-(x,t;T)\le u(x,t+T)\quad\mbox{in }\Omega^-\times\mathbb{R}_+,\qquad   u(x, t+T)\le u^+(x,t;T)\quad\mbox{in } \overline\Omega\times\mathbb{R}_+.
\end{equation}

Let $\phi^\pm$ be the unique functions satisfying respectively
$$
\Delta\phi^\pm=0\quad\text{in }\Omega^\pm, \qquad \phi^\pm=0 \quad\text{on }\partial\Omega^\pm,\qquad |\phi^\pm(x)-\log|x||\le C\quad\text{in }\Omega^\pm.
$$
We define $M^\pm(T)=\int_{\Omega^\pm}u^\pm(\cdot,0;T)\phi^\pm$. Note that
$M^-(T)\le M_\phi^*\le M^+(T)$ for all $T>0$. Since $0\le\phi^+-\phi^-\le C$, and $\phi^+(x)\le \log|x|+C$, using~\eqref{eq:decay.mass} we get
\[
|M^+(T)-M^-(T)|
\le \int_{\Omega^-}u(\cdot,T)(\phi^+-\phi^-)+\int_{B_R\cap\Omega} u(\cdot,T)\phi^+
\le C\int_\Omega u(\cdot,T)\le \frac{C}{\log T}.
\]
We conclude that $\lim_{T\to\infty}M^\pm(T)=M_\phi^*$.

Let $G^\pm(x,t;T)=\U(x,t;2mM^\pm(T)/\log t)$. Note that $G^-\le G\le G^+$. A direct computation shows that
$$
\begin{aligned}
0&\le (t\log t)^{\frac1m}\big(G^+(x,t;T)-G^-(x,t;T)\big)=F_{2mM^+(T)}(\widetilde\xi)-
F_{2mM^-(T)}(\widetilde\xi)\\
&\le C
\begin{cases}
(M^+(T))^{\frac1m}-(M^-(T))^{\frac1m},& 1<m\le 2,\\
\big((M^+(T))^{\frac{m-1}m}-(M^-(T))^{\frac{m-1}m}\big)^{\frac1{m-1}},& m>2.
\end{cases}
\end{aligned}
$$
On the other hand, by the results in~\cite{Gilding-Goncerzewicz-2007} for simply connected smooth holes,
$$
\lim_{t\to\infty}(t\log t)^{\frac1m}\sup_{x\in\mathcal{O}_\delta(t)}|u^\pm(x,t;T)-G^\pm(x,t;T)|=0.
$$
Hence, for every $\delta>0$, given $\varepsilon>0$, there exists $T_\varepsilon\ge0$ such that for every $T\ge T_\varepsilon$ we have
$$
\limsup_{t\to\infty}\,(t\log t)^{\frac1m}\sup_{x\in\mathcal{O}_\delta(t)}| u(x,t+T)-G(x,t)|\le \varepsilon.
$$

Now, using the monotonicity of $F_{*}$, we get
$$
\begin{array}{l}
\displaystyle(t\log t)^{\frac1m}\sup_{x\in\mathcal{O}_\delta(t)}\big|G(x,t-T_\varepsilon)-G(x,t)\big|\\
\displaystyle\qquad\le \Big(\Big(\frac{t\log t}{(t-T_\varepsilon)\log(t-T_\varepsilon)}\Big)^{\frac1m}-1\Big) F_{*}(0)+\sup_{|\tilde\xi|\ge \delta}\big|F_{*}(\tilde\xi g_\varepsilon(t))-F_{*}(\tilde\xi)\big|,
\end{array}
$$
where $\tilde \xi$ is as in \eqref{eq:limit.is.selfsimilar}, and $g_\varepsilon(t)=\Big(\frac{t}{t-T_\varepsilon}\Big)^{\frac1{2m}}\Big(\frac{\log(t-T_\varepsilon)}
{\log t}\Big)^{\frac{m-1}{2m}}$. Hence, since $F_{*}$ is uniformly continuous,
$$
\lim_{t\to\infty}(t\log t)^{\frac1m}\sup_{x\in\mathcal{O}_\delta(t)}\big|G(x,t-T_\varepsilon)-G(x,t)\big|=0.
$$

We next note that for $t$ large $\mathcal{O}_\delta(t)\subset \mathcal{O}_\delta(t-T_\varepsilon)$.  Hence,
$$
\begin{array}{l}
\displaystyle (t\log t)^{\frac1m}\sup_{x\in\mathcal{O}_\delta(t)}| u(x,t)-G(x,t)|\\
\displaystyle\quad\le
\Big(\frac{t\log t}{(t-T_\varepsilon)\log(t-T_\varepsilon)}\Big)^{\frac1m}((t-T_\varepsilon)\log (t-T_\varepsilon))^{\frac1m}\sup_{x\in\mathcal{O}_\delta(t-T_\varepsilon)}| u(x,t-T_\varepsilon+T_\varepsilon)-G(x,t-T_\varepsilon)|
\\
\displaystyle\qquad +
(t\log t)^{\frac1m}\sup_{x\in\mathcal{O}_\delta(t)}| G(x,t-T_\varepsilon)-G(x,t)|,
\end{array}
$$
which combined with the above estimates yields
$$
\limsup_{t\to\infty}\,
(t\log t)^{\frac1m}\sup_{x\in\mathcal{O}_\delta(t)}| u(x,t)-G(x,t)|\le \varepsilon.
$$
Letting $\varepsilon\to0$ we finally get the far field behavior~\eqref{eq:far.field.limit}, from where the behavior of the mass follows easily, using also~\eqref{eq:global.rate}.

The asymptotic behavior for the free boundary, formula~\eqref{eq:limit.support}, can be immediately deduced from the corresponding result for the case of simply connected smooth holes in~\cite{Gilding-Goncerzewicz-2007} and  estimate~\eqref{eq-comparison-asym}.
\end{proof}

\section{Control from above}
\label{sect:control.above}
\setcounter{equation}{0}

In this section we  construct a suitable supersolution of the porous medium equation in the inner region
$$
\mathcal{I}_{\delta,T}=\{(x,t): x\in\mathcal{I}_\delta(t),\ t\ge T\}.
$$
with $\delta>0$ small enough and $T$ big enough, and we use it to prove the \lq\lq upper part'' of Theorem~\ref{thm:main}. This supersolution is a perturbation of the asymptotic limit,
\begin{equation}
\label{eq-V}
\begin{array}{l}
V(x,t)=\eta c(t)G(x,t)w(x,t), \qquad
c(t)=1+\kappa_0\Big(\frac T t\Big)^{\mu},\qquad \eta>1,\ \kappa_0>0,\ \mu\in(0,1),
\\
w(x,t)=\left(\frac{(\phi(x))^{\nu(t)}+k}{\big(\frac{\log t}{2m}\big)^{\nu(t)}}\right)^{1/m},
\qquad \nu(t)=1-\frac1{\log t},\qquad k>0.
\end{array}
\end{equation}
We note that
\begin{equation}
\label{eq:def.A.B}
\begin{array}{l}
\partial_t V-\Delta V^m=\mathcal{A}+\mathcal{B}, \quad\text{where }\\[8pt]
\mathcal{A}=\eta c' G w+\eta c w \partial_t G-\eta^m c^m w^m\Delta G^m+\eta c G\partial_t w,\\[8pt]
\mathcal{B}=-\eta^m c^m G^m\Delta w^m-2\eta^m c^m\nabla w^m\cdot\nabla G^m.
\end{array}
\end{equation}
We will see that both $\mathcal{A}$ and $\mathcal{B}$ are nonnegative in the outer part of the inner region given by
$$
\mathcal{I}_{\delta,T}^{\mbox{\rm o}}=\Big\{(x,t)\in \mathcal{I}_{\delta,T}: |x|\ge t^{\frac1{2m}}/(\log t)^{2}\Big\}.
$$
However, though $\mathcal{B}$ is still positive in the inner part of the inner region
$$
\mathcal{I}_{\delta,T}^{\mbox{\rm i}}=\Big\{(x,t)\in \mathcal{I}_{\delta,T}: |x|\le t^{\frac1{2m}}/(\log t)^{2}\Big\},
$$
this is not necessarily the case for $\mathcal{A}$.
Hence, we will need to estimate carefully both terms to check that $\mathcal{B}$ dominates in this set.

\noindent\emph{Remark.  } Any power bigger than 2 for the logarithmic correction in the definition of the inner and the outer part of the inner region would also do the job.

\begin{lema}
\label{lem:estimates.A}
Let  $\mathcal{A}$ be defined by~\eqref{eq-V}--\eqref{eq:def.A.B}  and $\delta_*:= \xi_*\sqrt{(m-1)/(2m)}$. There are
a constant $\gamma=\gamma(m,M_\phi^*,k)>0$ and a time $T_{\mathcal{A}}=T_{\mathcal{A}}(\Omega,m,M_\phi^*,k)$ such that
\begin{equation}
\label{eq:estimate.A.below}
\mathcal{A}(x,t)\ge -\frac{\gamma \eta c(t)}{{t^{1+\frac1m}(\log t)^{\frac1m}}}\quad\text{in }\mathcal{I}_{\delta,T}\text{ for any }\delta\in(0,\delta_*)\text{ and }T\ge T_{\mathcal{A}}.
\end{equation}
Moreover, there are values $T_{\mathcal{A}}^{\rm o}=T_{\mathcal{A}}^{\rm o}(\eta,\Omega,m,M_\phi^*,k)$ and $\mu_*=\mu_*(\eta,m,M_\phi^*)>0$ such that $\mathcal{A}\ge 0$  in $\mathcal{I}_{\delta,T}^{\mbox{\rm o}}$ for all  $\delta\in(0,\delta_*)$ and $\mu\in(0,\mu_*)$ if $T\ge T_{\mathcal{A}}^{\rm o}$.
\end{lema}
\begin{proof} 
Straightforward computations show that, for any $\delta\in(0,\xi_*)$ and $T\ge \mbox{\rm e}$,
\begin{equation}
\label{eq:derivatives}
\begin{array}{l}
c'(t)\ge -\mu \frac{c(t)}t,
\\
\partial_t G(x,t)=\Delta G^m(x,t)-\frac{(2mM_\phi^*)^{\frac1m}}{(t\log t)^{1+\frac1m}}\big(1-\frac{|\tilde \xi|^2}{\xi_*^2}\big)_+^{\frac1{m-1}-1},\\
\Delta G^m(x,t)=-\frac1{m t^{1+\frac1m}(\log t)^{\frac1m}}\Big(\frac{M_\phi^*}{2\pi}\Big)^{\frac1m} \big(1-\frac{|\tilde \xi|^2}{\xi_*^2}\big)_+^{\frac1{m-1}-1}\big(1-{\textstyle\frac m{m-1}}\frac{|\tilde\xi|^2}{\xi_*^2}\big),\\
\partial_t w(x,t)=  \frac{w(x,t)}m\Big(\frac{(\phi(x))^{\nu(t)}\log\phi(x)}{\big((\phi(x))^{\nu(t)}+k\big)t(\log t)^{2}}
-\frac{\log\big(\frac{\log t}{2m}\big)}{t(\log t)^{2}}-\frac{\nu(t)}{t\log t}\Big),
\end{array}
\qquad\text{in }\mathcal{I}_{\delta,T}.
\end{equation}
Moreover, there is a time $T_w=T_w(\Omega,m,k)$ such that $w\le 2$ in $\mathcal{I}_{\delta,T}$ if $T\ge T_w$. Thus, there are positive constants $\gamma_1=\gamma_1(m,M_\phi^*,k)$ and $\gamma_2=\gamma_2(m,M_\phi^*)$  such that
\begin{equation}
\label{eq:estimates.w_t.Delta G^m}
\partial_t w(x,t)\ge -\frac{\gamma_1}{t\log t},\qquad
0>\Delta G^m(x,t)>-\frac{\gamma_2}{t^{1+\frac1m}(\log t)^{\frac1m}}\qquad\text{in }\mathcal{I}_{\delta,T},
\end{equation}
for any $\delta\in(0,\delta_*)$ and $T\ge T_w$. Since $0\le G(x,t)<F_*(0)/(t\log t)^{\frac1m}$, we easily get~\eqref{eq:estimate.A.below} if $T\ge T_{\mathcal{A}}\ge T_w$, with $T_{\mathcal{A}}=T_{\mathcal{A}}(\Omega,m,M_\phi^*,k)$ large enough.

The idea to prove the positivity of $\mathcal{A}$ in $\mathcal{I}_{\delta,T}^{\mbox{\rm o}}$ is that the good term $-\eta^m c^m w^m\Delta G^m$ is not only positive there, but also big enough to compensate for the negative leading order terms if $\mu$ is small. Indeed, since $\phi(x)\ge\log|x|-C$ and
\[
\log |x|\ge \frac{\log t}{2m}-2\log\log t\quad\mbox{in }\mathcal{I}_{\delta,T}^{\mbox{\rm o}},
\]
there exists a time $T_{\eta,m}$ depending only on $\eta$ and $m$ such that
\[
(\eta w(x,t))^{m-1}\ge 1+\frac{\eta^{m-1}-1}{2}\quad\text{for }t\ge T_{\eta,m}.
\]
Hence, taking $T\ge\max(T_{\mathcal{A}}, T_{\eta,m})$, since $c(t)>1$ for every $t\ge T$, we get
$$
\mathcal{A}(x,t)\ge \eta c(t) w(x,t)\Big(\frac{\frac{\gamma_2}2(\eta^{m-1}-1)-\mu F_*(0)}{t^{1+\frac1m}(\log t)^{\frac1m}}+o\Big(\frac1{t^{1+\frac1m}(\log t)^{\frac1m}}\Big)\Big)
\quad\text{in }\mathcal{I}_{\delta,T}^{\mbox{\rm o}}.
$$
The result follows just taking $\mu_*=\frac{\gamma_2}{4F_*(0)}(\eta^{m-1}-1)$, and then $T_{\mathcal{A}}^{\rm o}\ge \max(T_{\mathcal{A}}, T_{\eta,m})$ big enough.
\end{proof}

\begin{lema}
\label{lem:estimates.B}
Let  $\mathcal{B}$ be defined by~\eqref{eq-V}--\eqref{eq:def.A.B}. There is a time $T_{\mathcal{B}}^{\rm o}=T_{\mathcal{B}}^{\rm o}(\Omega,m)$ such that $\mathcal{B}\ge 0$ in $\mathcal{I}_{\delta,T}^{\mbox{\rm o}}(s)$ for all $T\ge T_{\mathcal{B}}^{\rm o}$ and $\delta\in(0,\xi_*)$. Moreover, there are a constant $\varpi=\varpi(m,M_\phi^*)$ and a time $T_{\mathcal{B}}^{\rm i}=T_{\mathcal{B}}^{\rm i}(m,M_\phi^*)$ such that
for all $\delta\in(0,\xi_*)$ and $T\ge T_{\mathcal{B}}^{\rm i}$,
\begin{equation}
\label{eq:estimate.B.below}
\displaystyle\mathcal{B}(x,t)\ge \frac{\varpi\eta c(t)}{t^{1+\frac1m}}
\quad\text{in }\mathcal{I}_{\delta,T}^{\mbox{\rm i}}.
\end{equation}
\end{lema}
\begin{proof}
Easy computations yield, for any $\delta\in(0,\xi_*)$ and $(x,t)\in\mathcal{I}_{\delta,T}$,
\begin{equation}
\label{eq:grad.laplac.Gw}
\begin{array}{l}
\nabla G^m(x,t)=-\frac1{2m t^{1+\frac1m}(\log t)^{\frac1m}}\Big(\frac{M_\phi^*}{2\pi}\Big)^{\frac1m} \big(1-\frac{|\tilde \xi|^2}{\xi_*^2}\big)_+^{\frac1{m-1}}x,
\\[8pt]
\nabla w^m(x,t)=\nu(t)\frac{(\phi(x))^{\nu(t)-1}}{(\frac{\log t}{2m})^{\nu(t)}}\nabla\phi(x),
\\[8pt]
\Delta w^m(x,t)=\nu(t)(\nu(t)-1)\frac{(\phi(x))^{\nu(t)-2}}{(\frac{\log t}{2m})^{\nu(t)}}
|\nabla \phi(x)|^2.
\end{array}
\end{equation}
Moreover, given $R$ as in estimate~\eqref{eq-bound-nabla-phi}, there is a time $T_{\mathcal{B}}^{\rm o}=T_{\mathcal{B}}^{\rm o}(\Omega,m)$  such that $|x|>R$ in $\mathcal{I}_{\delta,T}^{\mbox{\rm o}}$ for any $T\ge T_{\mathcal{B}}^{\rm o}$, which immediately implies, using that estimate, that  $\nabla w^m\cdot\nabla G^m\le0$ there. On the other hand,
$\Delta w^m\le 0$, since $\nu(t)\le1$, and hence $\mathcal{A}\ge0$.

In order to estimate the behavior in the inner part of the inner region, we notice that there is a time $T_\phi=T_\phi(m,M_\phi^*)$ such that $|\tilde\xi|<\frac{\xi_*}2$,
$ \phi(x)\le\frac{\log t}m$ and $\frac12\le\nu(t)< 1$ in $\mathcal{I}_{\delta,T}^{\rm i}$ if $T\ge T_\phi$.  Hence, using also that $|\nabla \phi(x)|$ is comparable to $1/|x|$ in $\Omega$, we easily get that there are positive constants $\varpi_1=\varpi_1(m,M_\phi^*)$ and $\varpi_2=\varpi_2(m,M_\phi^*)$ such that
$$
-(G^m\Delta w^m)(x,t)\ge\frac{\varpi_1}{t^{1+\frac1m}},\quad-2(\nabla G^m\cdot\nabla w^m)(x,t) \ge-\frac{\varpi_2}{(t\log t)^{1+\frac1m}}
\quad\text{in }\mathcal{I}_{\delta,T}^{\rm i},
$$
if $\delta\in(0,\xi_*)$ and $T\ge T_\phi$. Since $\eta, c(t)>1$, we finally get~\eqref{eq:estimate.B.below} if $T\ge T_{\mathcal{B}}^{\rm i}\ge T_\phi$, with $T_{\mathcal{B}}^{\rm i}=T_{\mathcal{B}}^{\rm i}(m,M_\phi^*)$ large enough.
\end{proof}

Combining the estimates in Lemmas~\ref{lem:estimates.A} and~\ref{lem:estimates.B},  we immediately get that $V$ is a supersolution of the equation. Note that $\delta_*<\xi_*$.
\begin{coro}\label{prop-supersolution}
Let  $V$ be defined by~\eqref{eq-V} and $\delta_*$ and $\mu_*$ as in Lemma~\ref{lem:estimates.A}. There is a time $T_*=T_*(\eta,\Omega,m,M_\phi^*,k)$ such that
\[
\partial_t V-\Delta V^m\ge0\quad\mbox{in }\mathcal{I}_{\delta,T},\text{ for all }\delta\in(0,\delta_*),\ \mu\in(0,\mu_*),\text{ and }T\ge T_*.
\]
\end{coro}

In order to prove that $u\le V$ in $I_{\delta,T}$, it is then just enough to prove that this inequality holds in the parabolic boundary of the set. The ordering in the outer boundary will come from the far field behavior. It is here  where  we are performing the matching.

\begin{prop}\label{prop-control from above} Let $u$ be a solution to \eqref{problem},  $\eta>1$, $V$ as in~\eqref{eq-V}, and $\delta_*$ and $\mu_*$ as in Lemma~\ref{lem:estimates.A}. Given $k>0$ and $\delta\in(0,\delta_*)$, there is a time $T^+=T^+(\eta,\delta,\Omega,m,M_\phi^*,k)>0$ such that for any $T\ge T^+$ there is a value $\kappa_0>0$ for which, for all $\mu\in(0,\mu_*)$, $u\le V$ in $\mathcal{I}_{\delta,T}$.
\end{prop}
\begin{proof}
We first note that $u=0\le V$ on $\partial\Omega\times\mathbb{R}_+$. As for the outer boundary, we have from~\eqref{eq:far.field.limit} that there exists $T_{\eta,\delta,m}>0$ such that
$$
u(x,t)\le G(x,t)+\frac{(\eta-1)F_*(\delta_*)}{2 (t\log t)^{\frac1m}}\le \Big(1+\frac{\eta-1}2\Big)G(x,t)\quad\text{for }|x|=\frac{\delta t^{\frac1{2m}}}{(\log t)^{\frac{m-1}{2m}}},\ t\ge T_{\eta,\delta,m}.
$$
On the other hand,  there is a time $T_{\Omega,\eta,\delta}>0$ such that
$$
\eta w(x,t)\ge 1+\frac{\eta-1}{2} \quad\text{for }|x|=\frac{\delta t^{\frac1{2m}}}{(\log t)^{\frac{m-1}{2m}}},\ t\ge T_{\Omega,\eta,\delta}.
$$
Since $c(t)\ge1$, we conclude that
\[
u(x,t)\le V(x,t)\quad\text{for }|x|=\frac{\delta t^{\frac1{2m}}}{(\log t)^{\frac{m-1}{2m}}},\ t\ge \max(T_{\eta,\delta,m},T_{\Omega,\eta,\delta}).
\]

Given $T\ge T^+:=\max(T_*,T_{\eta,\delta,m},T_{\Omega,\eta,\delta})$, where $T_*$ is the time given by Corollary~\ref{prop-supersolution},
\[
V(x,T)\ge\frac{(1+\kappa_0)k^{\frac1m}F_*(\delta_*)}{(T\log T)^{\frac1m}(\frac{\log T}{2m})^{\frac{\nu(T)}m}}\ge \|u_0\|_{L^\infty}\ge u(x,T),
\]
if $\kappa_0>0$ is big enough.
The result follows then from the comparison principle.
\end{proof}
The \lq\lq upper part'' of Theorem~\ref{thm:main} is now easy.
\begin{teo}
\label{thm:above}
Assume the hypotheses of Theorem~\ref{thm:main}. Given $\varepsilon>0$, there exists a time $T_\varepsilon$ such that,  for all $\delta\in(0,\delta_*)$,
\[
\mathcal{M}(x,t):=t^{\frac1m}(\log t)^{\frac2m}\frac{\Big(u(x,t)-\Big(\frac{2m\phi(x)}
{\log t}\Big)^{\frac1m}G(x,t)\Big)}{\big(\log(|x|+\mbox{ \rm e})\big)^{\frac1m}}\le\ep\quad\text{in }\mathcal{I}_{\delta,T_\varepsilon}.
\]
\end{teo}
\begin{proof}
We decompose $\mathcal{M}$ as
\begin{equation}
\label{eq:decomposition.M}
\begin{array}{l}
\mathcal{M}(x,t)=\frac{t^{\frac1m}(\log t)^{\frac2m}}{\big(\log(|x|+\mbox{\rm e})\big)^{\frac1m}}\Big(\big(u(x,t)-c(t)w(x,t)G(x,t)\big)+(c(t)-1)w(x,t)G(x,t)
\\[10pt]
\hskip1.7cm+\big(w(x,t)-\Big(\frac{2m\phi(x)}
{\log t}\Big)^{\frac{\nu(t)}m}\big)G(x,t)+\Big(\frac{2m\phi(x)}
{\log t}\Big)^{\frac1m}\big(\Big(\frac{2m\phi(x)}
{\log t}\Big)^{\frac{\nu(t)-1}m}-1\big)G(x,t)\Big).
\end{array}
\end{equation}
Let $\eta=1+\varepsilon$,
$k=\varepsilon^m$, and $\delta\in(0,\delta_*)$.
By Proposition \ref{prop-control from above}, given any $T\ge T^+$, there are  values $\kappa_0>0$  and $\mu\in(0,\mu_*)$   such that
\[
u(x,t)-c(t)w(x,t)G(x,t)\le \ep c(t) w(x,t)G(x,t)\quad\text{if }(x,t)\in\mathcal{I}_{\delta,T}.
\]
On the other hand, since $(a+b)^{\frac1m}-a^{\frac1m}\le b^{\frac1m}$ for all $a,b>0$ if $m>1$, we have
$$
w(x,t)-\Big(\frac{2m\phi(x)}
{\log t}\Big)^{\frac{\nu(t)}m}\le \Big(\frac{2m}
{\log t}\Big)^{\frac{\nu(t)}m}\ep.
$$
We finally observe that $(\log t)^{\frac{\nu(t)-1}m}\to1$ as $t\to\infty$, which in particular implies that
$$
0\le t^{\frac1m}(\log t)^{\frac2m}\frac{w(x,t)G(x,t)}{\big(\log(|x|+\mbox{ \rm e})\big)^{\frac1m}}\le C,\quad (x,t)\in \mathcal{I}_{\delta,T},
$$
since $\phi(x)\le \log|x|+C$ and $G(x,t)\le (t\log t)^{-\frac1m}F_*(0)$. Therefore, using also that $c(t)\to 1$, we conclude that there is a time $T_\ep\ge T^+$ such that
$\mathcal{M}(x,t)\le K\ep$  for all $t\ge T_\ep$ for some constant $K$ independent of $\ep$.
\end{proof}

\section{Control from below}
\label{sect:control.below}
\setcounter{equation}{0}

In this section we prove the \lq\lq lower part'' of Theorem~\ref{thm:main}.
The construction of the subsolution is a bit more involved than that of the supersolution. Comparison will be performed in the intersection of the inner set $\mathcal{I}_{\delta,T}$ with an approximation of $\Omega$. So, let us start by constructing this approximate set.

Let $r_0$ small, so that, on the one hand, the set ${\mathcal N}_{2r_0}=\{x\in\overline\Omega:\mbox{dist\,}\{x,\H\}\le 2r_0\}$ can be parametrized by $x=\bar x+s\textbf{n}(\bar x), \ \bar x\in\partial\H,\ 0\le s\le 2r_0$, where $\textbf{n}(\bar x)$ is the exterior unit normal to $\partial \H$ at $\bar x$, and on the other hand $\partial_{\textbf{n}}\phi>0$ in ${\mathcal N}_{2r_0}$.
Notice that
$$
\Gamma_0:=\{x=\bar x+r_0\textbf{n}(\bar x):\bar x\in\partial\H\}=\{x\in\overline\Omega: \mbox{dist\,}\{x,\H\}=r_0\}
$$
is a $C^{1,\alpha}$ surface.
Let $\bar\alpha_0=\inf_{\Gamma_0}\phi(x)$. Since $\phi(x)\to\infty$ as $|x|\to\infty$ and $\phi$ is harmonic in $\Omega$, then $\phi(x)\ge \bar\alpha_0$ if $\mbox{dist\,}\{x,\H\}\ge r_0$.

Let now $\alpha_0\in(0,\bar\alpha_0)$ and $
\D_{\alpha_0}=\{x\in\overline\Omega:\phi(x)< \alpha_0\}\cup \H$.
Then, $\partial \D_{\alpha_0}$  is a $C^\infty$ curve, since $\phi=\alpha_0$ and  $\nabla\phi\neq0$ on it. Moreover, $\overline\H\subset\D_{\alpha_0}$.
We will construct a subsolution $v$ in $\Omega_{\alpha_0}=\R^2\setminus\overline{\D_{\alpha_0}}$ of the form
\begin{equation}
\label{eq:definition.subsolution}
\begin{array}{l}
v(x,t)=\eta c(t)G(x,t)w(x,t), \qquad
c(t)=1-\kappa_0\Big(\frac Tt\Big)^{\mu},\qquad  \eta,\mu,\kappa_0\in(0,1),
\\
w(x,t)=\left(\frac{\phi(x)^{\nu(t)}-\alpha_0^{\nu(t)}}{\big(\frac{\log t}{2m}\big)^{\nu(t)}}\right)^{\frac1m},\qquad\nu(t)=1+\frac1{\log t}, \qquad\alpha_0\in(0,\bar \alpha_0).
\end{array}
\end{equation}
Notice that $v$ vanishes on $\partial\Omega_{\alpha_0}$, and that $\Omega_{\alpha_0}$ approaches $\Omega$ as $\alpha_0$ goes to 0. We will denote
$\mathcal{I}_{\delta,T,\alpha_0}=\{(x,t): x\in I_\delta(t)\cap\Omega_{\alpha_0},\ t\ge T\}$.

We will use the same kind of decomposition to check that $v$ is a subsolution as the one we used to deal with the supersolution $V$, namely
\begin{equation}
\label{eq:def.A.B.v}
\begin{array}{l}
\partial_t v-\Delta v^m=\mathcal{A}+\mathcal{B}, \quad\text{where }\\[8pt]
\mathcal{A}=\eta c' G w+\eta c w \partial_t G-\eta^m c^m w^m\Delta G^m+\eta c G\partial_t w,\\[8pt]
\mathcal{B}=-\eta^m c^m G^m\Delta w^m-2\eta^m c^m\nabla w^m\cdot\nabla G^m.
\end{array}
\end{equation}
The next two lemmas are devoted to obtain  estimates for $\mathcal{A}$ and $\mathcal{B}$ good enough as to prove that $v$ is a subsolution in adequate sets. As we had to do for the supersolution, we have to consider separately two regions, the outer part of the inner region, which in this case is given by
$$
\mathcal{I}_{\delta,T,\alpha_0}^{\mbox{\rm o}}=\Big\{(x,t)\in \mathcal{I}_{\delta,T,\alpha_0}: |x|\ge t^{\frac1{2m}}/\log t\Big\},
$$
and the inner part of the inner region,
$$
\mathcal{I}_{\delta,T,\alpha_0}^{\mbox{\rm i}}=\Big\{(x,t)\in \mathcal{I}_{\delta,T}: |x|\le t^{\frac1{2m}}/\log t\Big\}.
$$
As we will see, $\mathcal{A}$ is negative both in the outer and in the inner part of the inner region. However, though $\mathcal{B}$ is negative in the inner part of the inner region, this is not necessarily the case in the outer part, and hence we have to check that $\mathcal{A}$ dominates $\mathcal{B}$ there.
\begin{lema}
\label{lem:estimates.A.subsolution}
Let  $\mathcal{A}$ be defined by~\eqref{eq:definition.subsolution}--\eqref{eq:def.A.B.v}  and $\delta_*$ as in Lemma~\ref{lem:estimates.A}. There are
constants $\gamma=\gamma(m,M_\phi^*,\alpha_0)>0$ and $\mu_*=\mu_*(\eta,m,M_\phi^*,\kappa_0)>0$, and a time $T_{\mathcal{A}}=T_{\mathcal{A}}(\eta,m,M_\phi^*,\alpha_0)$ such that, for any $\delta\in(0,\delta_*)$, $\mu\in(0,\mu_*)$, and $T\ge T_{\mathcal{A}}$,
\begin{equation}
\label{eq:estimate.A.above}
\mathcal{A}(x,t)\le-\frac{\gamma \eta c(t)w(x,t)}{{t^{1+\frac1m}(\log t)^{\frac1m}}}\quad\text{in }\mathcal{I}_{\delta,T,\alpha_0}.
\end{equation}
\end{lema}
\begin{proof}
A straightforward computation shows that
$$
\partial_t w(x,t)=\frac{w(x,t)}m
\left(\frac{\log\big(\frac{\log t}{2m}\big)}{t(\log t)^2}-\frac{\nu(t)}{t\log t}-\frac{\log \alpha_0}{t(\log t)^2}-\frac{(\phi(x))^{\nu(t)}\big(\log\phi(x)-\log \alpha_0\big)}{t(\log t)^2((\phi(x))^{\nu(t)}-\alpha_0^{\nu(t)})}\right).
$$
Since $\nu(t)\ge 1$, and $\phi(x)> \alpha_0$ in $\Omega_{\alpha_0}$, we conclude that there is a time $T_{m,\alpha_0}$, depending only on $m$ and $\alpha_0$ such that
$$
\partial_t w(x,t)\le -\frac{w(x,t)}{2mt\log t}\quad\mbox{in } \mathcal{I}_{\delta,T,\alpha_0} \text{ if }T\ge T_{m,\alpha_0}.
$$

We next notice that
$$
c'(t)\le \mu^{\frac12}\frac{c(t)}t\quad\mbox{if } t\ge T \text{ and }\mu^{\frac12}\in\big(0,\frac{1-\kappa_0}{\kappa_0}\big).
$$
Moreover, since $\phi(x)\le \log|x|+C$, there exists a time $T_{\eta,m}$, depending only on $\eta$ and $m$, such that for all $T\ge
 T_{\eta,m}$,
\[
(\eta w(x,t))^{m-1}\le 1-\frac{1-\eta^{m-1}}{2}\quad\text{in }
\mathcal{I}_{\delta,T,\alpha_0}.
\]

We finally observe that there is a constant $\gamma_*=\gamma_*(m,M_\phi^*)>0$ such that
\[
\Delta G^m(x,t)\le -\frac{\gamma_*}{t^{1+\frac1m}(\log t)^{\frac1m}}\quad\text{in }I_{\delta,T}\text{ if }\delta\in(0,\delta_*);
\]
see~\eqref{eq:derivatives}. Hence, since $c(t)<1$ and $\partial_t G\le \Delta G^m$, taking $T\ge \max(T_{m,\alpha_0},T_{\eta,m})$  we conclude that
$$
\begin{aligned}
\mathcal{A}(x,t)&\le\eta c(t) w(x,t)\Big(\frac{\mu^{\frac12} G(x,t)}{t}+\frac{\Delta G^m(x,t)(1-\eta^{m-1})}2\Big)\\[8pt]
&\le\frac{\eta c(t) w(x,t)}{t^{1+\frac1m}(\log t)^{\frac1m}}\Big(\mu^{\frac12} F_*(0)-\frac{\gamma_*(1-\eta^{m-1})}2\Big),
\end{aligned}
$$
from where the result follows just taking $\mu_*=\Big(\min(\frac{\gamma_*(1-\eta^{m-1})}{4F_*(0)},
\frac{1-\kappa_0}{\kappa_0})\Big)^{2}$ and $\gamma=\frac{\gamma_*(1-\eta^{m-1})}{4}$.
\end{proof}

\begin{lema}
\label{lem:estimate.B.above}
Let  $\mathcal{B}$ be defined by~\eqref{eq:definition.subsolution}--\eqref{eq:def.A.B.v}. There are a constant $\varpi=\varpi(m,M_\phi^*)$ and a time $T_{\mathcal{B}}^{\rm o}=T_{\mathcal{B}}^{\rm o}(\Omega,m)$ such that
for all $\delta\in(0,\xi_*)$ and $T\ge T_{\mathcal{B}}^{\rm o}$,
\begin{equation}
\label{eq:estimate.B.above}
\displaystyle\mathcal{B}(x,t)\le \frac{\varpi\eta c(t)w(x,t)}{(t\log t)^{1+\frac1m}}
\quad\text{in }\mathcal{I}_{\delta,T,\alpha_0}^{\rm o}.
\end{equation}
Moreover, there is a time $T_{\mathcal{B}}^{\rm i}=T_{\mathcal{B}}^{\rm i}(\Omega,m,M_\phi^*)$ such that $\mathcal{B}\le 0$ in $\mathcal{I}_{\delta,T,\alpha_0}^{\rm i}$ for all $T\ge T_{\mathcal{B}}^{\rm i}$ and $\delta\in(0,\xi_*)$.
\end{lema}

\begin{proof}
We have~\eqref{eq:grad.laplac.Gw}, where now $\phi$ is as in~\eqref{eq:definition.subsolution}. In particular  $\Delta w^m\ge0$, since $\nu(t)\ge1$. Moreover, since $|x||\nabla \phi(x)|\le C$ in $\overline\Omega$, there is a constant $\varpi_1=\varpi_1(m,M_\phi^*)$ such that
\begin{equation}\label{eq-nablas2}
-2(\nabla G^m\cdot\nabla w^m)(x,t)\le \frac{\varpi_1}{t^{1+\frac1m}(\log t)^{\frac1m}\phi(x)}\Big(\frac{2m\phi(x)}{\log t}\Big)^{\nu(t)}\quad\text{in }\mathcal{I}_{\delta,T,\alpha_0}.
\end{equation}
On the other hand, since $\phi(x)\ge\log|x|-C$, there is a time $T^{\rm o}_\phi=T^{\rm o}_\phi(\Omega,m)$ such that
$\phi(x)\ge\frac{\log t}{4m}$ in $\mathcal{I}_{\delta,T,\alpha_0}^{\rm o}$  if $T\ge T^{\rm o}_\phi$. In particular, there is a large enough time $T_{\mathcal{B}}^{\rm 0}=T_{\mathcal{B}}^{\rm 0}(\Omega,m)\ge T^{\rm o}_\phi$ such that
$$
 w(x,t)\ge \Big(\frac12\Big)^{\frac{\nu(t)}m}\Big(\frac{2m\phi(x)}{\log t}\Big)^{\frac{\nu(t)}m}
\ge \frac12\Big(\frac{2m\phi(x)}{\log t}\Big)^{{\nu(t)}}\quad\text{in }\mathcal{I}_{\delta,T,\alpha_0}^{\rm o}\text{ if }T\ge T_{\mathcal{B}}^{\rm 0},
$$
and~\eqref{eq:estimate.B.above} follows easily with $\varpi=8m\varpi_1$, since $c(t),\eta\le 1$.

As for the inner part of the inner region, we observe that there is a time $T_\phi=T_\phi(\Omega,m,M_\phi^*)$ such that $\phi(x)\le \frac{\log t}m$ in $\mathcal{I}_{\delta,T,\alpha_0}$,   and hence in $\mathcal{I}_{\delta,T,\alpha_0}^{\rm i}$,  for $T\ge T_\phi$ and $\delta\in(0,\xi_*)$. Therefore,  since $|\nabla\phi(x)|$ is comparable to $1/|x|$ in $\Omega$,   there is a constant $\varpi_2$ such that
\begin{equation*}\label{eq-Gamma Delta w}
-(G^m\Delta w^m)(x,t)\le  -\frac{\varpi_2}{t^{1+\frac1m}\phi(x)}  \Big(\frac{2m\phi(x)}{\log t}\Big)^{\nu(t)}\quad\text{in }\mathcal{I}_{\delta,T,\alpha_0}^{\rm i},
\end{equation*}
which combined with~\eqref{eq-nablas2} yields $\mathcal{B}\le0$ if $T\ge T_{\mathcal{B}}^{\rm o}\ge T_\phi$, for some $T_{\mathcal{B}}^{\rm o}=T_{\mathcal{B}}^{\rm o}(\Omega,m,M_\phi^*)$ large enough.
\end{proof}
Combining the estimates in Lemmas~\ref{lem:estimates.A.subsolution} and~\ref{lem:estimate.B.above}, we immediately get that $v$ is a subsolution of the equation.
\begin{coro}\label{prop-subsolution}
Let $v$ be defined by~\eqref{eq:definition.subsolution} and $\delta_*$ and $\mu_*$ as in Lemma~\ref{eq:estimate.A.above}. There is a time $T_*=T_*(\eta,\delta,\Omega,M_\phi^*,\alpha_0)$ such that
\[
\partial_t v-\Delta v^m\le0\quad\mbox{in }\mathcal{I}_{\delta,T,\alpha_0}\text{ for all }\delta\in(0,\delta_*),\ \mu\in(0,\mu_*),\text{ and }T\ge T_*.
\]
\end{coro}
In order to prove that $u\ge v$ in $I_{\delta,T,\alpha_0}$, it is then just enough to prove that this inequality holds in the parabolic boundary of the set. The ordering in the outer boundary will come from the far field behavior. It is here  where  we are performing the matching.
\begin{prop}\label{prop-control from below} Let $u$ be a solution to \eqref{problem},  $\eta<1$, $v$ as in~\eqref{eq:definition.subsolution}, and $\delta_*$ and $\mu_*$ as in Lemma~\ref{lem:estimates.A.subsolution}. Given  $\alpha_0\in(0,\bar\alpha_0)$ and $\delta\in(0,\delta_*)$, there is a time $T^-=T^-(\eta,\delta,\Omega,m,M_\phi^*,\alpha_0)>0$ such that for any $T\ge T^-$ there is a value $\kappa_0>0$ for which, for all $\mu\in(0,\mu_*)$, $u\ge v$ in $\mathcal{I}_{\delta,T,\alpha_0}$.
\end{prop}
\begin{proof} We first note that $v=0\le u$ on $\partial \Omega_{\alpha_0}\times\mathbb{R}_+$. As for the outer boundary, we have from~\eqref{eq:far.field.limit} that there exists a time $T_{\eta,\delta,m}>0$ such that
$$
u(x,t)\ge G(x,t)-\frac{(1-\eta)F_*(\delta_*)}{2 (t\log t)^{\frac1m}}\ge \Big(1-\frac{1-\eta}{2}\Big)G(x,t)\quad\text{for }|x|=\frac{\delta t^{\frac1{2m}}}{(\log t)^{\frac{m-1}{2m}}},\ t\ge T_{\eta,\delta,m}.
$$
On the other hand,  there is a time $T_{\Omega,\eta,\delta}>0$ such that
$$
\eta w(x,t)\le 1-\frac{1-\eta}{2} \quad\text{for }|x|=\frac{\delta t^{\frac1{2m}}}{(\log t)^{\frac{m-1}{2m}}},\ t\ge T_{\Omega,\eta,\delta}.
$$
Since $c(t)\le 1$, we conclude that
\[
u(x,t)\ge v(x,t)\quad\text{for }|x|=\frac{\delta t^{\frac1{2m}}}{(\log t)^{\frac{m-1}{2m}}},\ t\ge \max(T_{\eta,\delta,m},T_{\Omega,\eta,\delta}).
\]

By~\eqref{eq:limit.support}, given $\delta\in(0,\xi_*)$, there is a time $T_{\xi_*}=T_{\xi_*}(\delta,\Omega,m,\alpha_0)$  such that $u(x,T)\ge\ell>0$  if $x\in \Omega_{\alpha_0}\cap \mathcal{I}_\delta(T)$, $T\ge T_{\xi_*}$. Thus, given $T\ge T^-:=\max(T_*,T_{\eta,\delta,m},T_{\Omega,\eta,\delta},T_{\xi_*})$, where $T_*$ is the time given by Corollary~\ref{prop-subsolution},
\[
v(x,T)\le\frac{(1-\kappa_0)(1+\eta)F_*(0)}{2(T\log T)^{\frac1m}}\le \ell\le u(x,T)\quad\mbox{in }\mathcal{I}_{\delta,T,\alpha_0},
\]
if $\kappa_0\in(0,1)$ is close enough to 1.
The result follows then from the comparison principle.
\end{proof}
 The \lq\lq lower part'' of Theorem~\ref{thm:main} is now easy.
\begin{teo}
Assume the hypotheses of Theorem~\ref{thm:main} and let $\mathcal{M}$ be as in Theorem~\ref{thm:above}. Given $\varepsilon>0$, there exists a time $T_\varepsilon$ such that $\mathcal{M}\ge -\ep$ in $\mathcal{I}_{\delta,T_\varepsilon}$.
\end{teo}
\begin{proof}
Given $\ep>0$ small, we take $\alpha_0=\frac{\ep^m}{2m(F_*(0))^m}$.
Then,
\[
\frac{t^{\frac1m}(\log t)^{\frac2m}}{\big(\log(|x|+\mbox{\rm e})\big)^{\frac1m}}
\Big(\frac{2m\phi(x)}
{\log t}\Big)^{\frac1m}G(x,t)\le\ep\quad\text{in }(\Omega\setminus\Omega_{\alpha_0})\times\mathbb{R}_+.
\]
Hence, it is enough to prove that there exists a time $T_\ep$ such that $\mathcal{M}\ge -\ep$ in $\mathcal{I}_{\delta,T_\ep,\alpha_0}$. To this aim, we use the decomposition~\eqref{eq:decomposition.M}, where now $c$ and $w$ are given by~\eqref{eq:definition.subsolution}.

Let $\eta=1-\varepsilon$ and $\delta\in(0,\delta_*)$.
By Proposition~\ref{prop-control from below}, given any $T\ge T^-$, there are  values $\kappa_0>0$  and $\mu\in(0,\mu_*)$   such that
\[
u(x,t)-c(t)w(x,t)G(x,t)\ge -\ep c(t) w(x,t)G(x,t)\quad\text{if }(x,t)\in\mathcal{I}_{\delta,T,\alpha_0}.
\]
On the other hand, since $(a-b)^{\frac1m}-a^{\frac1m}\ge -b^{\frac1m}$ for all $a\ge b>0$ if $m>1$, we have
$$
w(x,t)-\Big(\frac{2m\phi(x)}
{\log t}\Big)^{\frac{\nu(t)}m}\ge - (\log t)^{-\frac{\nu(t)}{m}}
\Big(\frac{\ep}{F_*(0)}\Big)^{\nu(t)}\ge -K\ep
$$
for some constant $K$ independent of $\ep$, assuming, without loss of generality,  $T\ge\textrm{e}$ and $\ep\le1$.

We conclude as in the proof of Theorem~\ref{thm:above}.
\end{proof}


\end{document}